\documentclass[12pt,reqno]{amsart}

\title[]%{Commutator Estimates for Negative Fractional Laplacian on Bounded Domains and Applications to Electroconvection}
{Entropy Analysis of some Active Scalar Equations}

\author[E. Abdo]{Elie Abdo}
\address[E. Abdo]
{
    Department of Mathematics \\
    American University of Beirut \\
    Beirut 1107-2020\\
    Lebanon.
} 
\email{ea94@aub.edu.lb}

\usepackage{xfrac}

\usepackage[margin=1in]{geometry}
\usepackage{amsmath, amsthm, amssymb, yhmath, xfrac, mathtools}
\usepackage[most]{tcolorbox}
\usepackage[normalem]{ulem}

\usepackage{times}
\usepackage{color}
\usepackage{hyperref}
\newcommand{\pa}{\partial}
\newcommand{\la}{\label}

\newcommand{\na}{\nabla}
\newcommand{\be}{\begin{equation}}
\newcommand{\ee}{\end{equation}}
\newcommand{\ba}{\begin{array}{l}}
\newcommand{\ea}{\end{array}}

\newtheorem{thm}{Theorem}[section]
\newtheorem{prop}[thm]{Proposition}
\newtheorem{lemma}[thm]{Lemma}
\newtheorem{cor}[thm]{Corollary}

\usepackage{cite}

\usepackage{MnSymbol,wasysym}
\usepackage{accents}

\newcommand{\N}{\mathbb N}

\newcommand{\R}{\mathbb R}

\def\RR{{\mathbb R}}

\date{\today}
\begin{document}
\begin{abstract}
We study a broad family of active scalar equations with fractional dissipation on $\mathbb{R}^2$ and prove that their entropy diverges to $-\infty$ at a sharp rate of order $\log t$. The proof is based on a suitable regularization scheme, uniform entropy estimates, and a limiting argument via the Vitali convergence theorem. Our analysis introduces new fractional logarithmic Sobolev inequalities, weighted commutator estimates, and fractional moment bounds. These results provide a general framework for studying spatial decay and long-time dynamics of nonlocal nonlinear partial differential equations.
\end{abstract}


\noindent

\maketitle

\section{Introduction}
We consider a family of active scalar equations 
\be \label{activescalar}
\pa_t \theta + u \cdot \na \theta + \kappa \Lambda^{\alpha}\theta = 0
\ee on $\R^2$, with initial data $\theta(x,0) = \theta_0$ and with decay  at infinity. 
Here, the vector field $u:\mathbb{R}^2\to\mathbb{R}^2$ is given by either of the following two constitutive laws:

\noindent\textbf{(I) Linear velocity law.}
The velocity field is of the form $u=T(\theta)$, where $T$ is a linear operator satisfying the following properties:

\begin{enumerate}
\item[(i)] \textbf{Divergence-free condition:}
$\nabla\cdot T(\theta)=0$.

\item[(ii)] \textbf{$L^p$-boundedness:}
For every $p\in(1,\infty)$,
$\|T(\theta)\|_{L^p}\le C_p\|\theta\|_{L^p}$.

\item[(iii)] \textbf{Lipschitz continuity in $L^p$:}
For every $p\in(1,\infty)$,
$\|T(\theta_1)-T(\theta_2)\|_{L^p}
\le C_p\|\theta_1-\theta_2\|_{L^p}$.
\end{enumerate}

\noindent\textbf{(II) Bilinear velocity law.}
The velocity field is of the form $u=\tilde T(\theta,\theta)$, where $\tilde T$ is a bilinear operator satisfying the following properties: 

\begin{enumerate}
\item[(i)] \textbf{Divergence-free condition:}
$\nabla\cdot \tilde T(\theta,\theta)=0$.

\item[(ii)] \textbf{Bilinear $L^p$ estimate:}
For every $1<p,q,r<\infty$ satisfying $\frac1r=\frac1p+\frac1q$,
$\|\tilde T(\theta,\theta)\|_{L^r}
\le C_{p,q,r}\|\theta\|_{L^p}\|\theta\|_{L^q}$.

\item[(iii)] \textbf{Lipschitz stability estimate:}
For every $1<p,q,r<\infty$ satisfying $\frac1r=\frac1p+\frac1q$,
$\|\tilde T(\theta_1,\theta_1)-\tilde T(\theta_2,\theta_2)\|_{L^r}
\le C_{p,q,r}\|\theta_1-\theta_2\|_{L^p}
\bigl(\|\theta_1\|_{L^q}+\|\theta_2\|_{L^q}\bigr)$.
\end{enumerate}

The fractionally dissipative surface quasi-geostrophic (SQG) equation is an example of a velocity field of type {\bf{(I)}}, with $u=\mathcal{R}^{\perp}\theta$, where $\mathcal{R}$ denotes the Riesz transform \cite{constantin1999behavior}, whereas electroconvection in porous media provides an example of type ${\bf{(II)}}$, with $u=-\mathbb{P}(\theta \mathcal{R}\theta)$ and $\alpha = 1$, where $\mathbb{P}$ denotes the Leray projection onto the space of divergence-free vector fields \cite{AbdoIgnatova2023}.

In this paper, we investigate the long-time behavior of the relative entropy 
\be 
\mathcal{E}(t)=\int_{\mathbb{R}^2}\theta\log\theta\,dx
\ee associated with \eqref{activescalar}. The study of entropy is particularly natural in the context of active scalar equations, where the transport velocity is determined nonlocally by the scalar itself, leading to a delicate interplay between transport, nonlocal interactions, and dissipation. Such equations arise in a variety of important models that have been extensively studied in the literature, including the fractionally dissipative surface quasi-geostrophic (SQG) equation \cite{buckmaster2019nonuniqueness, constantin2001critical, constantin2016critical,  constantin2020estimates, constantin2018inviscid, constantin2023global, constantin1994formation, constantin2018local, constantin2015long, constantin2012nonlinear, constantin1999behavior, constantin2008regularity, ignatova2019construction, zelati2016global, kiselev2007global, stokols2020holder}, models of electroconvection in porous media \cite{AbdoIgnatova2023, Abdo2026Electroconvection}, and other active scalar systems \cite{abdo2024regularity, castro2009incompressible, dabkowski2014global, hassainia2021kam,  hmidi2017existence, isett2015holder, shvydkoy2011convex}. For several of these models, it has been shown that weak solutions converge to zero as $t\to\infty$ on the whole space \cite{constantin1999behavior}. However, convergence of the scalar to zero does not by itself capture the manner in which its mass spreads over the whole space. The entropy provides a complementary perspective on this phenomenon: its behavior quantifies the progressive spatial spreading and delocalization of the scalar. In this work, we show that, despite the convergence $\theta(t)\to0$, the relative entropy exhibits a strikingly different behavior, namely, $\mathcal{E}(t)\to-\infty$ as $t\to\infty$. More precisely, we establish the sharp logarithmic rate $\mathcal{E}(t)\sim -C\log t$, thereby providing a quantitative description of the large-time spreading of the scalar and revealing information about the asymptotic dynamics that is not captured by the pointwise or norm convergence $\theta(t)\to0$.

\subsection{Main theorems} We consider the regularized family of active-scalar equations 
\be \label{moll}
\pa_t \theta^{\epsilon} + J_{\epsilon}u^{\epsilon} \cdot \na \theta^{\epsilon} + \kappa \Lambda^{\alpha} \theta^{\epsilon} - \epsilon \Delta \theta^{\epsilon} = 0
\ee with initial data
\be \label{moll1}
\theta^{\epsilon}(x,0)  = J_{\epsilon} \theta_0
\ee and decay at infinity. Here $J_{\epsilon}\theta_0 = \rho_\epsilon * \theta_0$ is the standard mollifier operator defined as a convolution operator with $\rho_\varepsilon(x)
=
\frac{1}{\epsilon^2}
\rho\!\left(\frac{x}{\varepsilon}\right)$, where $\rho$ is a nonnegative, smooth and  compactly supported function, obeying $\int_{\mathbb{R}^2} \rho(x)\,dx = 1$. The existence of unique global smooth solutions to these $\epsilon$-systems is standard and follows from classical fixed-point arguments. We study the entropy behavior associated with these regularized systems:

\begin{thm} \label{t1}
Let $\alpha \in (0,2)$. Let $\eta \in \left(\alpha - \frac{1}{2}, \frac{\alpha}{2} \right) \cap (0,1)$ if $\alpha \in (0,1)$, and $\eta \in \left(\frac{\alpha}{2} - \frac{1}{2}, \frac{\alpha}{4} \right) \cap (0,1)$ if $\alpha \in [1,2)$. 
Suppose $\theta_0$ is positive on $\R^2$ and obeys the following regularity assumptions:
\be 
\theta_0 \in L^1 \cap L^{\infty}, \hspace{1cm}  (1+|x|^2)^{\frac{1}{2}+\eta} \theta_0 \in L^2(\R^2), \hspace{1cm} \left|\int_{\R^2} \theta_0 \log \theta_0 dx\right| <\infty.   
\ee Then the family of regularized solutions $\left\{\theta^{\epsilon} \right\}_{\epsilon > 0}$ to \eqref{moll}--\eqref{moll1} is nonnegative and is uniformly bounded in $\epsilon$ in the spaces 
\be 
L^{\infty}(0,\infty; L^p(\RR^2)) \cap L^2(0, \infty; \dot{H}^{\frac{\alpha}{2}}(\R^2))
\ee for any $p \ge 1$. Moreover, there are positive constants $\Gamma, c_1, \dots, c_6$ that depends only on the initial data and a positive integer $m$ (independent of $\epsilon$) such that 
\be 
\|(1+|x|^2)^{\frac{1}{2}+\eta} \theta^{\epsilon}(t)\|_{L^2}^2 \le \Gamma (1+t)^m
\ee for any $t \ge 0$ and $\epsilon \in (0,1)$, and 
\be 
-c_1 \log (c_2+c_3t) \le \mathcal{E}^{\epsilon}(t):= \int_{\R^2} \theta^{\epsilon} \log \theta^{\epsilon} dx \le -c_4 \log (c_5+c_6t)
\ee for any $t \ge 0$ and $\epsilon \in (0,1)$.
\end{thm}

Passing to the limit as $\epsilon \rightarrow 0$, we obtain:

\begin{thm} \label{t2}
Let $\alpha \in (0,2)$. Let $\eta \in \left(\alpha - \frac{1}{2}, \frac{\alpha}{2} \right) \cap (0,1)$ if $\alpha \in (0,1)$, and $\eta \in \left(\frac{\alpha}{2} - \frac{1}{2}, \frac{\alpha}{4} \right) \cap (0,1)$ if $\alpha \in [1,2)$.
Suppose $\theta_0$ is positive on $\R^2$ and obeys the following regularity assumptions:
\be 
\theta_0 \in L^1 \cap L^{\infty}, \hspace{1cm}  (1+|x|^2)^{\frac{1}{2}+\eta} \theta_0 \in L^2(\R^2), \hspace{1cm} \left|\int_{\R^2} \theta_0 \log \theta_0 dx\right| <\infty.   
\ee  Then the family of active scalar equations \eqref{activescalar} has a weak solution $\theta$ that is nonnegative almost everywhere, and that satisfies the regularity criteria 
\be 
\theta \in L^{\infty}(0,\infty; L^p(\RR^2)) \cap L^2(0, \infty; \dot{H}^{\frac{\alpha}{2}}(\R^2))
\ee for any $p \ge 1$. Moreover, there are constants $c_1, \dots, c_6$ that depends only on the initial data such that 
\be 
-c_1 \log (c_2+c_3t) \le \mathcal{E}(t):= \int_{\R^2} \theta \log \theta dx \le -c_4 \log (c_5+c_6t)
\ee for any $t \ge 0$.
\end{thm}

\subsection{Methodology and challenges} For the inviscid active scalar equation, the entropy is conserved in time due to the transport structure of the model. The presence of fractional dissipation completely changes this behavior: the entropy is no longer conserved and, in fact, diverges to $-\infty$ at a sharp rate of order $\log t$. Establishing this phenomenon is highly nontrivial. Our approach begins with the construction of a suitable regularization scheme, for which we prove uniform entropy growth estimates, independent of the regularization parameter. We then pass to the limit and obtain weak solutions to the original family of active scalar equations. The entropy behavior of these weak solutions is recovered through a delicate application of the Vitali convergence theorem.

A fundamental preliminary step is to prove that nonnegative initial data remain nonnegative for all times. The main difficulty, however, lies in deriving sharp upper and lower bounds for the entropy. To obtain the upper bound, we study the evolution of the entropy functional $e^{-c\mathcal E}$, for a suitably chosen constant $c>0$. By establishing quantitative lower bounds for this functional, we derive the desired entropy estimates. The argument relies on a new fractional logarithmic Sobolev inequality on the whole space, which is proved in this work.

The derivation of the lower entropy bounds is considerably more challenging and is based on a careful analysis of fractional moments. Obtaining such moment estimates is delicate due to the nonlocal nature of the fractional dissipation. To overcome this difficulty, we construct commutators generated by interchanging fractional powers of the weight $|x|$ with fractional powers of the Laplacian and establish sharp commutator estimates when the dissipation exponent $\alpha\in(0,1)$. More generally, we develop doubly weighted commutator estimates associated with the interchange of the product of two H\"older continuous functions and fractional powers of the Laplacian. These estimates are derived through a refined application of the Hardy--Littlewood--Sobolev inequality.

For the regime $\alpha\in[1,2)$, we introduce a new decomposition method that effectively reduces the order of the fractional dissipation, transforming the problem into one involving exponents in the range $(0,1)$. This reduction allows us to analyze the resulting commutators using the techniques developed for the lower-order case. In addition, we investigate the commutativity properties between fractional powers of $|x|$ and mollification operators, which constitute a key ingredient in obtaining uniform estimates for the regularized system.

To the best of our knowledge, the techniques developed in this paper are entirely new. Beyond the specific active scalar models considered here, they provide a general framework for studying spatial decay and moment propagation for a broad class of nonlocal nonlinear partial differential equations.

\subsection{Organization of the paper} The paper is organized as follows. In Section~\ref{s2}, we derive all the commutator estimates needed in the sequel. In Section~\ref{s3}, we establish uniform fractional moment bounds. In Section~\ref{s4}, we prove Theorem~\ref{t1}, while in Section~\ref{s5}, we prove Theorem~\ref{t2}.

Throughout the paper, we denote by $C$ a positive universal constant that may depend on the parameters of the system, and this constant may change from line to line in the proofs.

\section{Weighted Commutator Estimates} \label{s2}

For $\alpha \in (0,2)$, the fractional Laplacian of order $\alpha$ on $\R^2$ can be represented by the integral representation formula
\be \la{fracint}
\Lambda^{\alpha} f(x) = c_{\alpha} P.V. \int_{\R^2} \frac{f(x) - f(y)}{|x-y|^{2+\alpha}} dy, 
\ee 
and its inverse power, the Riesz potential, is given by 
\be \la{rieszpotential}
\Lambda^{-\alpha} f(x) = C_{\alpha}\int_{\R^2} \frac{f(y)}{|x-y|^{2-\alpha}} dy.
\ee Here $c_{\alpha}$ and $C_{\alpha}$ are constants that depend only on $\alpha$. We drop the P.V. notation for simplicity.

The goal of this section is to derive weighted commutator estimates. We start by stating and proving a crucial interpolation inequality, which will be frequently used in this section:

\begin{lemma} Let 
\be 
\beta \in (0,1), \hspace{0.5cm} \delta > \max \left(0; \beta - \frac{1}{2} \right), \hspace{0.5cm} \gamma \in \left(\beta, \frac{1}{2}+\delta \right).
\ee  Let $f \in L^2$ such that $(\sqrt{1+|x|^2})^{\frac{1}{2}+\delta}f \in L^2$. Then there exists a positive constant $C$ depending on $\beta, \gamma, \delta$ such that 
\be \label{lem14}
\|\Lambda^{-\beta} f\|_{L^2} 
\le C_{\beta, \gamma, \delta} \|(\sqrt{1+|x|^2})^{\frac{1}{2}+\delta} f\|_{L^2}^{\frac{\beta}{\gamma}} \|f\|_{L^2}^{1-\frac{\beta}{\gamma}}.
\ee 
\end{lemma}

\begin{proof}
By the two-dimensional Hardy-Littlewood-Sobolev lemma, we have 
\be \label{lem11}
\|\Lambda^{-\beta}f\|_{L^2} 
\le C\|f\|_{L^{\frac{2}{1+\beta}}}
\ee for $\beta \in (0,1)$. Moreover, since $\frac{2}{1+\beta} \in (\frac{2}{1+\gamma} ,2)$, we can use the log-convexity of $L^p$ norms and interpolate to bound the latter by 
\be \label{lem12}
\|f\|_{L^{\frac{2}{1+\beta}}}
\le C\|f\|_{L^{\frac{2}{1+\gamma}}}^{\frac{\beta}{\gamma}} \|f\|_{L^2}^{1-\frac{\beta}{\gamma}}.
\ee 
In view of H\"older's inequality with exponents $1+\gamma$ and $\frac{\gamma + 1}{\gamma}$, it holds that 
\be 
\begin{aligned}
\|f\|_{L^{\frac{2}{1+\gamma}}}^{\frac{2}{1+\gamma}}
&= \int_{\R^2} \frac{|f|^{\frac{2}{1+\gamma}}(\sqrt{1+|x|^2})^{\frac{1+2\delta}{1+\gamma}}}{(\sqrt{1+|x|^2})^{\frac{1+2\delta}{1+\gamma}}}
dx
\\&\le \left(\int_{\R^2} (\sqrt{1+|x|^2})^{1+2\delta} |f|^2 dx \right)^{\frac{1}{1+\gamma}} \left(\int_{\R^2} \frac{1}{(\sqrt{1+|x|^2})^{\frac{1+2\delta}{\gamma}}} dx \right)^{\frac{\gamma}{\gamma+1}},
\end{aligned}
\ee which boils down to 
\be \label{lem13}
\|f\|_{L^{\frac{2}{1+\gamma}}} \le C\|(\sqrt{1+|x|^2})^{\frac{1}{2}+\delta} f\|_{L^2}
\ee due to the integrability of the rational  functions $(1+|x|)^{-q}$ on $\R^2$ for any $q > 2$. Indeed, in our case, we have 
\be 
\gamma < \frac{1}{2}+\delta \Rightarrow 2\gamma < 1+2\delta \Rightarrow \frac{1+2\delta}{\gamma} > 2. 
\ee Putting \eqref{lem11}, \eqref{lem12} and \eqref{lem13} together yields the desired estimate \eqref{lem14}.  
\end{proof}

Now we are ready to derive doubly weighted commutator estimates:

\begin{prop} \label{double} Let 
\be 
\alpha \in (0,1), \hspace{2cm} \eta \in \left(\alpha - \frac{1}{2}, \frac{\alpha}{2} \right) \cap (0,1).
\ee 
Let $\omega_1, \omega_2$ be two H\"older continuous functions of order $1/2 +\eta$. Let $f$ be a  scalar function obeying
\be 
f, \omega_1 f, \omega_2f, \left(\sqrt{1+|x|^2}\right)^{\frac{1}{2}+\eta}  \omega_1 f, \left(\sqrt{1+|x|^2}\right)^{\frac{1}{2}+\eta} \omega_2 f \in L^2, \hspace{0.5cm} f \in L^1. 
\ee  Then there is a positive real number $M$ depending on $\alpha$ and $\eta$ such that the following commutator estimate
\be \label{prop14}
\begin{aligned}
\|\Lambda^{\alpha}( \omega_1 \omega_2 f) - \omega_1 \omega_2 \Lambda^{\alpha}f\|_{L^2} 
&\le
r \left(\|f\|_{L^1} + \|\left(\sqrt{1+|x|^2}\right)^{\frac{1}{2}+\eta} \omega_1 f\|_{L^2} + \|\left(\sqrt{1+|x|^2}\right)^{\frac{1}{2}+\eta} \omega_2 f\|_{L^2}\right)
\\&\quad\quad+ C_{\alpha,  \eta, \omega_1, \omega_2} r^{-M} \left(\|f\|_{L^2} + \|\omega_1 f\|_{L^2} + \|\omega_2 f\|_{L^2}\right)
\end{aligned}
\ee holds for any $r > 0$. 
In particular, if $\omega_1(x) = \omega_2(x) = (\sqrt{1+|x|^2})^{\frac{1}{2}+\eta}$, then there is a positive real number $N$ that depends on $\alpha$ and $\eta$ such that 
\be \label{corre}
\begin{aligned}
&\|\Lambda^{\alpha}\left( \left(\sqrt{1+|x|^2}\right)^{1+2\eta} f\right) - \left(\sqrt{1+|x|^2}\right)^{1+2\eta}\Lambda^{\alpha} f\|_{L^2}
\\&\le
r \left(\|f\|_{L^1} + \|\left(\sqrt{1+|x|^2}\right)^{1+2\eta} f\|_{L^2} \right)
+ C_{\alpha,  \eta} r^{-N} \|f \|_{L^2}.
\end{aligned}
\ee holds for any $r > 0$. 
\end{prop}

\begin{proof}
We denote by $\mathcal{C}$ the commutator 
\be 
\mathcal{C} = \Lambda^{\alpha}(\omega_1 \omega_2 f) - \omega_1 \omega_2 \Lambda^{\alpha}f.
\ee Using the integral representation formula of the fractional Laplacian \eqref{fracint},  we rewrite $\mathcal{C}$ as 
\be 
\mathcal{C}(x) = c_{\alpha} \int_{\R^2} \frac{\omega_1(x) \omega_2(x) - \omega_1(y) \omega_2(y)}{|x-y|^{2+\alpha}} f(y)dy. 
\ee We decompose $\mathcal{C}$ into the sum of three terms, $\mathcal{C}_1$, $\mathcal{C}_2$, and $\mathcal{C}_3$, where
\be 
\mathcal{C}_{1}(x) = c_{\alpha} \int_{\R^2} \frac{\omega_1(y) (\omega_2(x) -\omega_2(y)}{|x-y|^{2+{\alpha}}} f(y)dy,
\ee 
\be 
\mathcal{C}_{2}(x) = c_{\alpha} \int_{\R^2} \frac{\omega_2(y) (\omega_1(x) - \omega_1(y))}{|x-y|^{2+{\alpha}}} f(y)dy, 
\ee and 
\be 
\mathcal{C}_{3}(x) = c_{\alpha} \int_{\R^2} \frac{(\omega_1(x)- \omega_1(y)) (\omega_2(x) - \omega_2(y))}{|x-y|^{2+{\alpha}}} f(y)dy. 
\ee We now estimate $\mathcal{C}_{1}$,  $\mathcal{C}_{2}$, and $\mathcal{C}_3$ separately. In view of the H\"older continuity property 
\be 
|\omega_2(x) - \omega_2(y)| \le C_{\omega_2}|x-y|^{\frac{1}{2}+\eta} 
\ee that holds for all $x, y \in \R^2$, we have 
\be 
|\mathcal{C}_1(x)|
\le c_{\alpha}C_{\omega_2} \int_{\R^2} \frac{|\omega_1(y)f(y)|}{|x-y|^{\frac{3}{2}+\alpha-\eta}} dy = c_{\alpha}C_{\omega_2} C_{\alpha}^{-1}(\Lambda^{-\frac{1}{2}+\alpha - \eta} |\omega_1 f|)(x). 
\ee Since $\beta = \frac{1}{2} - \alpha + \eta \in (0,1)$ and $\eta > \max \left(0; \beta - \frac{1}{2}\right)$, we can apply  the interpolation inequality \eqref{lem14} and further bound $\mathcal{C}_1$ in $L^2$ as follows,
\be 
\|\mathcal{C}_1\|_{L^2} 
\le C_{\eta, \alpha, \omega_2}  \|(\sqrt{1+|x|^2})^{\frac{1}{2}+\eta} \omega_1 f\|_{L^2}^{\frac{\beta}{\gamma}} \|\omega_1 f\|_{L^2}^{1-\frac{\beta}{\gamma}},
\ee where $\gamma \in \left(\beta, \frac{1}{2}+\eta \right)$. Due to Young's inequality for products, the latter reduces to 
\be \label{prop31}
\|\mathcal{C}_1\|_{L^2} 
\le  r \|(\sqrt{1+|x|^2})^{\frac{1}{2}+\eta} \omega_1 f\|_{L^2} + C_{\eta, \alpha, \omega_2}r^{-\frac{\beta}{\gamma - \beta}} \|\omega_1 f\|_{L^2}.
\ee As for $\mathcal{C}_2$, it has the same structure as $\mathcal{C}_1$ and can thus be estimated similarly, yielding
\be \label{prop32}
\|\mathcal{C}_2\|_{L^2} 
\le  r \|(\sqrt{1+|x|^2})^{\frac{1}{2}+\eta} \omega_1 f\|_{L^2} + C_{\eta, \alpha, \omega_2}r^{-\frac{\beta}{\gamma - \beta}} \|\omega_2 f\|_{L^2}.
\ee Finally, in order to estimate $\mathcal{C}_3$, we exploit the H\"older continuity property obeyed by $\omega_1$ and $\omega_2$ to bound it by 
\be 
|\mathcal{C}_3| \le c_{\alpha} C_{\omega_1} C_{ \omega_2} \int_{\R^2} \frac{|f(y)|}{|x-y|^{1+{\alpha}-2\eta}} dy =  c_{\alpha} C_{\omega_1} C_{ \omega_2} C_{\alpha}^{-1} (\Lambda^{-1+\alpha -2\eta} |f|)(x) ,
\ee which, after applying the Hardy-Littlewood-Sobolev inequality, interpolating, and making use of Young's inequality, yields  the $L^2$ bound
\be \label{prop33} 
\|\mathcal{C}_3\|_{L^2} \le  C_{\alpha, \omega_1, \omega_2} \|f\|_{\frac{2}{2-\alpha+2\eta}}
\le r \|f\|_{L^1} + C_{\alpha, \omega_1, \omega_2}r^{-{\frac{1-\alpha + 2\eta}{\alpha - 2\eta}}} \|f\|_{L^2} 
\ee for any $r > 0$.
Putting \eqref{prop31}, \eqref{prop32}, and \eqref{prop33} together yields the desired bound \eqref{prop14}. As for \eqref{corre}, it follows from \eqref{prop14} with $\omega_1 = \omega_2 = \left(\sqrt{1+|x|^2}\right)^{\frac{1}{2} + \eta}$ and the Cauchy-Schwarz type inequality
\be 
\begin{aligned}
\|\left(\sqrt{1+|x|^2}\right)^{\frac{1}{2} + \eta}f\|_{L^2}^2 
&\le \|\left(\sqrt{1+|x|^2}\right)^{1 + 2\eta}f\|_{L^2} \|f \|_{L^2} 
\\&\le r^2\|\left(\sqrt{1+|x|^2}\right)^{1 + 2\eta}f\|_{L^2}^2 + Cr^{-2}\|f\|_{L^2}^2.
\end{aligned}
\ee 
\end{proof}

As a consequence, we obtain the following corollary covering the $\alpha$-range $[1,2)$:

\begin{cor} \label{coro}
Let 
\be 
\alpha \in (0,2), \hspace{1cm} \eta \in \left(\frac{\alpha}{2} - \frac{1}{2}, \frac{\alpha}{4} \right) \cap (0,1), \hspace{1cm} b(x) = \left(\sqrt{1+|x|^2} \right)^{1+2\eta}.
\ee 
Let $f$ be a  scalar function obeying
\be 
f, bf, \Lambda^{\frac{\alpha}{2}}f, \Lambda^{\frac{\alpha}{2}}(bf) \in L^2, \hspace{0.5cm} f \in L^1. \ee  Then there is a positive constant $N$ that depends on $\alpha$ and $\eta$ such that the following commutator estimate
\be \label{prop15}
\begin{aligned}
&\|\Lambda^{\frac{\alpha}{2}}\left(b \Lambda^{\frac{\alpha}{2}}f\right)  - b\Lambda^{\frac{\alpha}{2}} (\Lambda^{\frac{\alpha}{2}} f)\|_{L^2}
\\&\quad\quad\le C_{\eta}r^2\left(\|f\|_{L^1} + \|bf\|_{L^2}\right) + C_{\alpha,\eta}r^{1-N}\|f\|_{L^2} + C_{\eta}r\|\Lambda^{\frac{\alpha}{2}} (bf)\|_{L^2} + C_{\alpha, \eta}r^{-N}\|\Lambda^{\frac{\alpha}{2}}f\|_{L^2}
\end{aligned}
\ee holds for any $r > 0$. 
\end{cor}

\begin{proof}
Substituting $f$ by $\Lambda^{\frac{\alpha}{2}}f$ in \eqref{corre}, we obtain 
\be \label{prop88}
\begin{aligned}
\|\Lambda^{\frac{\alpha}{2}}\left(b \Lambda^{\frac{\alpha}{2}}f\right)  - b\Lambda^{\frac{\alpha}{2}} (\Lambda^{\frac{\alpha}{2}} f)\|_{L^2}
\le
r \left(\|\Lambda^{\frac{\alpha}{2}}f\|_{L^1} + \|b \Lambda^{\frac{\alpha}{2}}f\|_{L^2} \right)
+ C_{\alpha,  \eta} r^{-N} \|\Lambda^{\frac{\alpha}{2}}f \|_{L^2}.
\end{aligned}
\ee The $L^1$ norm of $\Lambda^{\frac{\alpha}{2}}$ can be treated as follows,
\be 
\|\Lambda^{\frac{\alpha}{2}}f\|_{L^1} 
\le \left(\int_{\R^2} \left(\sqrt{1+|x|^2} \right)^{2+ 4\eta }|\Lambda^{\frac{\alpha}{2}}f|^2 dx  \right)^{\frac{1}{2}} \left(\int_{\R^2} \frac{1}{\left(\sqrt{1+|x|^2} \right)^{2+ 4\eta }} dx \right)^{\frac{1}{2}} \le C_{\eta} \|b\Lambda^{\frac{\alpha}{2}}f\|_{L^2},
\ee in which case \eqref{prop88} reduces to 
\be \label{prop77}
\begin{aligned}
\|\Lambda^{\frac{\alpha}{2}}\left(b \Lambda^{\frac{\alpha}{2}}f\right)  - b\Lambda^{\frac{\alpha}{2}} (\Lambda^{\frac{\alpha}{2}} f)\|_{L^2}
\le
C_{\eta} r \|b \Lambda^{\frac{\alpha}{2}}f\|_{L^2} 
+ C_{\alpha,  \eta} r^{-N} \|\Lambda^{\frac{\alpha}{2}}f \|_{L^2}.
\end{aligned}
\ee 
Subtracting and adding $\Lambda^{\frac{\alpha}{2}} (b f)$, we have
\be 
\|b\Lambda^{\frac{\alpha}{2}}f\|_{L^2}
\le \|b\Lambda^{\frac{\alpha}{2}}f - \Lambda^{\frac{\alpha}{2}}(bf) \|_{L^2} + \|\Lambda^{\frac{\alpha}{2}}(bf)\|_{L^2}
\ee 
 where the first term is a commutator that can be controlled through an application of the doubly weighted estimate \eqref{corre}, yielding
\be \label{prop99}
\begin{aligned}
&\|b \Lambda^{\frac{\alpha}{2}} f\|_{L^2} \le  r\|f\|_{L^1} + r\|bf\|_{L^2} + C_{\alpha,\eta}r^{-N}\|f\|_{L^2} + \|\Lambda^{\frac{\alpha}{2}}(bf)\|_{L^2}.
\end{aligned}
\ee  Combining the above estimates gives \eqref{prop15}. 
\end{proof}

Next, we establish a simpler  single weighted commutator estimate:

\begin{prop} \label{commutator-estimate-1}
Let $\omega$ be a Lipschitz function on $\R^2$. Let $f \in L^1 \cap L^2$. Let 
\be 
\alpha \in (0,1), \hspace{0.5cm} p \in \left(\frac{2}{1+\alpha}, \frac{2}{\alpha}\right).
\ee Then the commutator estimate
\be \label{propp}
\|\Lambda^{\alpha}(\omega f) - \omega \Lambda^{\alpha}f\|_{L^p}
\le C_{\omega}C_{\alpha, p} (\|f\|_{L^1} + \|f\|_{L^2}). 
\ee holds, where $C_{\omega}$ is the Lipschitz constant associated with $\omega.$
\end{prop}

\begin{proof}
 We denote by $\mathcal{K}$ the commutator
 \be 
\mathcal{K} = \Lambda^{\alpha}(\omega f) - \omega \Lambda^{\alpha}f.
 \ee In view of the integral representation formula \eqref{fracint} for the positive powers of the fractional Laplacian, we can rewrite $\mathcal{K}$ as 
 \be 
\mathcal{K}(x) = c_{\alpha}\int_{\R^2}\frac{\omega(x) - \omega(y)}{|x-y|^{2+\alpha}} f(y) dy. 
 \ee By the Lipschitz continuity property obeyed by $\omega$, there is a positive constant $C_{\omega}$ such that 
 \be 
|\omega(x) - \omega(y)| \le C_{\omega}|x-y|
 \ee for all $x, y \in \R^2$, and consequently, the commutator $\mathcal{K}$ can be bounded pointwise by 
 \be 
|\mathcal{K}(x)| \le C_{\omega} c_{\alpha}\int_{\R^2}\frac{|f(y)|}{|x-y|^{1+\alpha}} dy
=  C_{\omega} c_{\alpha}C_{\alpha}^{-1}(\Lambda^{\alpha - 1}|f|)(x).
 \ee An application of the Hardy-Littlewood-Sobolev lemma gives rise to 
 \be 
\|\mathcal{K}\|_{L^p} \le C_{\omega} C_{\alpha,p} \|f\|_{L^{\frac{2p}{2+(1-\alpha)p}}},
 \ee which, after interpolating between $L^1$ and $L^2$ and applying Young's inequality, boils down to 
 \be 
\|\mathcal{K}\|_{L^p} \le C_{\omega} C_{\alpha,p} \left(\|f\|_{L^1} + \|f\|_{L^2} \right).
 \ee 
\end{proof}

We end this section by proving some weighted commutator estimates involving mollifiers:

\begin{prop} Let $\omega_1, \omega_2$ be two H\"older continuous functions with respective exponents $\ell_1$ and $\ell_2$. Let $\epsilon > 0, p \ge 1$. Suppose $f, \omega_1 f, \omega_2 f \in L^p$. Then it holds that 
\be \label{molcom1}
\|\omega_1 \omega_2 J_{\epsilon}f - J_{\epsilon}(\omega_1 \omega_2 f)\|_{L^p}
\le C\epsilon^{\ell_1 + \ell_2} \|f\|_{L^p} + C\epsilon^{\ell_2} \|\omega_1 f\|_{L^p} + C\epsilon^{\ell_1} \|\omega_2 f\|_{L^p}.
\ee  If $\omega_2 = 1$, then we have 
\be \label{malcom5}
\|\omega_1 J_{\epsilon}f - J_{\epsilon}(\omega_1  f)\|_{L^p}
\le C\epsilon^{\ell_1} \|f\|_{L^p}.
\ee 
\end{prop}

\begin{proof}
For $x \in \R^2$, we have 
\be 
\begin{aligned}
(\omega_1 \omega_2 J_{\epsilon}f - J_{\epsilon}(\omega_1 \omega_2 f))(x)
&= \int_{\R^2} (\omega_1(x) \omega_2(x) - \omega_1(y) \omega_2(y)) \rho_{\epsilon}(x-y) f(y) dy 
\\&= \mathcal{B}_1 + \mathcal{B}_2 + \mathcal{B}_3
\end{aligned}
\ee where 
\be 
\mathcal{B}_1 = \int_{\R^2} (\omega_1(x) - \omega_1(y)) (\omega_2(x) - \omega_2(y)) \rho_{\epsilon}(x-y) f(y) dy, 
\ee 
\be 
\mathcal{B}_2 = \int_{\R^2} (\omega_2(x)- \omega_2(y)) \omega_1(y) \rho_{\epsilon}(x-y) f(y) dy 
\ee and 
\be 
\mathcal{B}_3 = \int_{\R^2} (\omega_1(x)- \omega_1(y)) \omega_2(y) \rho_{\epsilon}(x-y) f(y) dy. 
\ee By the H\"older continuity property obeyed by $\omega_1$ and $\omega_2$ and the nonnegativity of $\rho_{\epsilon}$, there is a positive constant $C$ such that 
\be 
|\mathcal{B}_1| + |\mathcal{B}_2| + |\mathcal{B}_3|
\le C |x|^{\ell_1 + \ell_2} \rho_{\epsilon} * |f| + C|x|^{\ell_2}\rho_{\epsilon} * |\omega_1 f| + C|x|^{\ell_1}\rho_{\epsilon} * |\omega_2f|,
\ee where $*$ denotes the convolution operator. By Young's convolution inequality, it follows that 
\be \label{molcom2}
\begin{aligned}
&\|\omega_1 \omega_2 J_{\epsilon}f - J_{\epsilon}(\omega_1 \omega_2 f)\|_{L^p}
\\&\quad\quad\le C\||x|^{\ell_1+ \ell_2} \rho_{\epsilon}\|_{L^1} \|f\|_{L^p}
+ C\||x|^{\ell_1} \rho_{\epsilon}\|_{L^1} \|\omega_2 f\|_{L^p}
+ 
 C\||x|^{\ell_2} \rho_{\epsilon}\|_{L^1} \|\omega_1 f\|_{L^p}.
 \end{aligned}
\ee Changing variables and using the facts that $\rho$ is compactly supported and integrable, we estimate 
\be \label{molcom3}
\||x|^{\ell} \rho_{\epsilon}\|_{L^1}
= \int_{\R^2} |x|^{\ell} \epsilon^{{-2}}\rho(x/\epsilon) dx
= \epsilon^{\ell} \int_{\R^2} |z|^{\ell} \rho(z) dz \le C\epsilon^{\ell},
\ee for $\ell = \ell_1, \ell_2, \ell_1 + \ell_2$. Putting \eqref{molcom2} and \eqref{molcom3} together, we obtain \eqref{molcom1}. If $\omega_2 = 1$, then $\mathcal{B}_1 = \mathcal{B}_2 = 0$, and $\mathcal{B}_3$ boils down to 
\be 
\mathcal{B}_3 = \int_{\R^2} (\omega_1(x) - \omega_1(y)) \rho_{\epsilon}(x-y) f(y) dy,
\ee which can be estimated similarly, yielding \eqref{malcom5}. 
\end{proof}

\section{Uniform Moment Bounds} \label{s3}

\begin{prop} \label{propsec3} Let $\alpha \in (0,2)$. Let $\eta \in \left(\alpha - \frac{1}{2}, \frac{\alpha}{2} \right) \cap (0,1)$ if $\alpha \in (0,1)$, and $\eta \in \left(\frac{\alpha}{2} - \frac{1}{2}, \frac{\alpha}{4} \right) \cap (0,1)$ if $\alpha \in [1,2)$. 
Suppose $\theta_0$ is positive on $\R^2$ and obeys the following regularity assumptions:
\be 
\theta_0 \in L^1 \cap L^{\infty}, \hspace{1cm}  (1+|x|^2)^{\frac{1}{2}+\eta} \theta_0 \in L^2(\R^2).   
\ee Then the family of regularized solutions $\left\{\theta^{\epsilon} \right\}_{\epsilon > 0}$ to \eqref{moll}--\eqref{moll1} is nonnegative and is uniformly bounded in $\epsilon$ in the spaces 
\be 
L^{\infty}(0,\infty; L^p(\RR^2)) \cap L^2(0, \infty; \dot{H}^{\frac{\alpha}{2}}(\R^2))
\ee for any $p \ge 1$. Moreover, there is a positive constant $\Gamma$ that depends only on the initial data and a positive constant $m$ that depends on $\alpha$ and $\eta$ such that  
\be 
\|(1+|x|^2)^{\frac{1}{2}+\eta} \theta^{\epsilon}(t)\|_{L^2}^2 \le \Gamma (1+t)^m
\ee for any $t \ge 0$.
\end{prop}

\begin{proof}
    The proof is divided into several major steps.

    {\bf{Step 1. Nonnegativity of the regularized solutions.}} Let $\epsilon > 0$. Fix $t \ge 0$. Since $\theta^{\epsilon}$ decays to $0$ as $|x|  \rightarrow \infty$, there exists a radius $R=R(\epsilon, t)$ that depends on $\epsilon$ and $t$ such that $|\theta^{\epsilon}(x,t)| \le 1$ for any $|x| \ge R$. Moreover, $\theta^{\epsilon}$ is a continuous function of $x$ on the compact ball centered at $0$ with radius $R$, and it is thus bounded there. Consequently, the spatial infimum of $\theta^{\epsilon}(\cdot, t)$ on $\R^2$ is finite.
    Let \be 
m^{\epsilon}(t) = \inf \left\{\theta^{\epsilon}(x,t): x \in \R^2\right\}.
    \ee If $m^{\epsilon}(t) \ge  0$, then $\theta^{\epsilon}(x,t) \ge 0$ for all $x \in \R^2$ and there is nothing to prove. Suppose $m^{\epsilon}(t) < 0$ at some time $t>0$. In view of the spatial decay of the regularized solution $\theta^{\epsilon}$, there exists an $R = R(\epsilon ,t ) > 0$ such that $|\theta^{\epsilon}(x,t)| \le \frac{|m^{\epsilon}(t)|}{2}$ for all $|x| \ge R$, in which case we have $\theta^{\epsilon} \ge m^{\epsilon}(t)/2$ for all $|x| \ge R$. Since $m^{\epsilon}(t)$ is negative, it follows that $m^{\epsilon}(t)/2 > m^{\epsilon}(t)$. In other words, $m^{\epsilon}(t)/2$ is a lower bound for $\theta^{\epsilon}$ on $|x| \ge R$. Hence 
    \be 
    m^{\epsilon}(t) = \inf \left\{\theta^{\epsilon}(x,t): |x| \le R \right\}.
    \ee However, $\theta^{\epsilon}$ is continuous on $|x| \le R$, so it attains its minimum there. Therefore, there is $\bar{x}$ such that $|\bar{x}| \le  R$ and \be 
\theta^{\epsilon}(\bar{x}, t) = \min \left\{\theta^{\epsilon}(x,t): |x| \le R \right\}= m^{\epsilon}(t). 
    \ee At $\bar{x}$, it holds that $\na \theta^{\epsilon}(\bar{x},t) = 0$ and $\Delta \theta^{\epsilon}(\bar{x},t) \ge 0$, thus
    \be 
(\pa_t \theta^{\epsilon})(\bar{x},t) + \kappa (\Lambda^{\alpha}\theta^{\epsilon})(\bar{x}, t) \ge 0.  
    \ee 
From the integral representation formula for the fractional Laplacian \eqref{fracint}, we note that $\Lambda^{\alpha}\theta^{\epsilon}(\bar{x}, t) \le 0$ and deduce that 
\be 
(\pa_t \theta^{\epsilon})(\bar{x}, t) \ge 0.
\ee Finally, by the Hamilton's trick \cite{mantegazza2011lecture}, it follows that 
\be 
\pa_t m^{\epsilon} \ge 0,
\ee which, after integrating in time from $0$ to $t$, gives
\be 
m^{\epsilon}(t) \ge m^{\epsilon} (0) > 0,
\ee contradicting our assumption that $m^{\epsilon}(t)$ is negative. Therefore, $\theta^{\epsilon}  \ge 0$ for any $x \in \R^2$ and any $t \ge 0$.

{\bf{Step 2. Uniform bounds in $L^p$ spaces.}} Let $p >1$.
The $L^p$ evolution of $\theta^{\epsilon}$ is described by 
\be \label{l2}
\frac{1}{p}\frac{d}{dt}\|\theta^{\epsilon}\|_{L^p}^p + \kappa \int_{\RR^2} (\theta^{\epsilon})^{p-1} \Lambda^{\alpha} \theta^{\epsilon} dx - \epsilon \int_{\R^2} 
(\theta^{\epsilon})^{p-1} \Delta \theta^{\epsilon} dx = 0.
\ee The nonlinear term in $u^{\epsilon}$ vanishes due to the divergence-free property obeyed by $J_{\epsilon} u^{\epsilon}$. In view of the C\'ordoba-C\'ordoba inequality \cite{cordoba2004maximum}, the diffusive and parabolic terms are nonnegative. Integrating in time from $0$ to $t$, we deduce that 
\be \label{ppnorm}
\|\theta^{\epsilon}(t)\|_{L^p} \le \|\theta_0\|_{L^p}
\ee for any $t \ge 0$. Moreover, it follows from \eqref{l2} with $p=2$ that
\be \label{l2e}
\kappa \int_{0}^{t} \left(\|\Lambda^{\frac{\alpha}{2}}\theta^{\epsilon}\|_{L^2}^2 + \epsilon \|\na \theta^{\epsilon}\|_{L^2}^2 \right) ds \le \frac{1}{2}\|\theta_0\|_{L^2}^2
\ee for any $t \ge 0$. As for the case $p=1$, it is evident that the $L^1$ norm of $\theta^{\epsilon}$ is conserved in time and obeys 
\be 
\|\theta^{\epsilon}(t)\|_{L^1} = \|\theta_0\|_{L^1}
\ee for any $t \ge 0$. 

{\bf{Step 3. Quadratic moment uniform bounds for $\alpha \in (0,1)$.}}
Let $a(x) = \sqrt{1+|x|^2}$. The $L^2$ norm of $a\theta^{\epsilon}$ evolves according to 
\be 
\frac{1}{2}\frac{d}{dt}\|a\theta^{\epsilon}\|_{L^2}^2
+ \int_{\RR^2} a J_{\epsilon}u^{\epsilon} \cdot \na \theta^{\epsilon} a \theta^{\epsilon} dx + \kappa \int_{\R^2} a \theta^{\epsilon} a \Lambda^{\alpha}\theta^{\epsilon} dx - \epsilon \int_{\R^2} a \theta^{\epsilon} a \Delta \theta^{\epsilon} dx  = 0.
\ee 
Subtracting and adding $\Lambda^{\alpha} (a\theta^{\theta})$, the diffusion term can be rewritten as 
\be 
\kappa \int_{\R^2} a \theta^{\epsilon} a \Lambda^{\alpha}\theta^{\epsilon} dx
= \kappa \|\Lambda^{\alpha}(a\theta^{\epsilon})\|_{L^2}^2 + \kappa \int_{\R^2} a\theta^{\epsilon} [a, \Lambda^{\alpha}]\theta^{\epsilon}  dx
\ee where $[a, \Lambda^{\alpha}]\theta^{\epsilon}$ denotes the commutator $a \Lambda^{\alpha} \theta^{\epsilon} - \Lambda^{\alpha} (a\theta^{\epsilon})$. In view of Proposition \ref{commutator-estimate-1}, we have 
\be 
\left| \kappa \int_{\R^2} a\theta^{\epsilon} [a, \Lambda^{\alpha}]\theta^{\epsilon}  dx \right|
\le C\|a\theta^{\epsilon}\|_{L^2} \left(\|\theta^{\epsilon}\|_{L^1} + \|\theta^{\epsilon}\|_{L^2} \right)
\le C\|a\theta^{\epsilon}\|_{L^2} \left(\|\theta_0\|_{L^1} + \|\theta_0\|_{L^2} \right).
\ee Regarding the nonlinear term in $u^{\epsilon}$, we can subtract the term $J_{\epsilon}u^{\epsilon} \cdot \na (a\theta^{\epsilon}) a \theta^{\epsilon}$ from the integrand at no cost due to the cancellation law  
\be 
\int_{\R^2} J_{\epsilon}u^{\epsilon} \cdot \na (a\theta^{\epsilon}) a \theta^{\epsilon} = 0.
\ee Thus, it follows that 
\be 
\begin{aligned}
&\left|\int_{\RR^2} a J_{\epsilon}u^{\epsilon} \cdot \na \theta^{\epsilon} a \theta^{\epsilon} dx \right|
= \left|\int_{\RR^2} a(\theta^{\epsilon})^2 J_{\epsilon}u^{\epsilon} \cdot \na a dx \right|
\\&\quad\quad\le C\|a\theta^{\epsilon}\|_{L^2}\|\theta^{\epsilon}\|_{L^4}\|u^{\epsilon}\|_{L^4}
\le C\Gamma_1 \|a\theta^{\epsilon}\|_{L^2},
\end{aligned}
\ee where $\Gamma_1$ is a positive constant depending only on the initial data. More precisely, $\Gamma_1 = \|\theta_0\|_{L^8}^2\|\theta_0\|_{L^4}$ if $u^{\epsilon}= \tilde{T}(\theta^{\epsilon}, \theta^{\epsilon})$ and $\Gamma_1 = \|\theta_0\|_{L^4}^2$ if $u^{\epsilon}= T\theta^{\epsilon}$. As for the $\epsilon$-regularizing term, we have 
\be 
\begin{aligned}
- \epsilon \int_{\R^2} a \theta^{\epsilon} a \Delta \theta^{\epsilon} dx  
&= \epsilon \|\na (a\theta^{\epsilon})\|_{L^2}^2
+ \epsilon \int_{\R^2} a \theta^{\epsilon} \left(\Delta (a\theta^{\epsilon}) - a\Delta \theta^{\epsilon} \right) dx
\\&= \epsilon \|\na (a\theta^{\epsilon})\|_{L^2}^2
+ \epsilon \int_{\R^2} a \theta^{\epsilon} \left(2\na a \cdot \na \theta^{\epsilon} + \theta^{\epsilon} \Delta a \right) dx.
\end{aligned}
\ee We bound the second term above as follows,  
\be 
\left|\epsilon \int_{\R^2} a \theta^{\epsilon} \left(2\na a \cdot \na \theta^{\epsilon} + \theta^{\epsilon} \Delta a \right) dx\right|
\le C\epsilon \|a\theta^{\epsilon}\|_{L^2} \left(\|\na \theta^{\epsilon}\|_{L^2} + \|\theta^{\epsilon}\|_{L^4} \right).
\ee Here we have used the fact that $\na a = \frac{x}{\sqrt{1+|x|^2}}\in L^{\infty}$ and $\Delta a  = \frac{2+|x|^2}{(1+|x|^2)^{3/2}} \in L^4$. Combining the above estimates, we obtain 
\be \label{1mB}
\frac{1}{2}\frac{d}{dt}\|a\theta^{\epsilon}\|_{L^2}^2 + \kappa \|\Lambda^{\alpha} (a\theta^{\epsilon})\|_{L^2}^2 +\epsilon \|\na (a\theta^{\epsilon})\|_{L^2}^2
\le C\|a\theta^{\epsilon}\|_{L^2} \left(\Gamma_2 + \epsilon \|\na \theta^{\epsilon}\|_{L^2}\right) 
\ee for some constant $\Gamma_2$ depending only on the initial data. The latter gives rise to the differential inequality 
\be 
\frac{d}{dt}\|a\theta^{\epsilon}\|_{L^2} \le C\left(\Gamma_2 + \epsilon \|\na \theta^{\epsilon}\|_{L^2}^2 + 1\right).
\ee Integrating in time from $0$ to $t$ and using \eqref{l2e} and \eqref{malcom5}, we infer that 
\be \label{atheta}
\|a\theta^{\epsilon}\|_{L^2} \le \|aJ_{\epsilon}\theta_0 - J_{\epsilon}(a\theta_0) ||_{L^2} + \|J_{\epsilon}(a\theta_0)\|_{L^2} + C(\Gamma_2+1)t + \epsilon \int_0^t \|\na \theta^{\epsilon}\|_{L^2}^2 ds \le \Gamma_3 (1+t)
\ee for any $t \ge 0$, where $\Gamma_3$ is a positive constant depending only on the initial data. Moreover, integrating \eqref{1mB} from $0$ to $t$ gives 
\be \label{epsat}
\kappa \int_{0}^{t} \|\Lambda^{\alpha}(a\theta^{\epsilon})(s)\|_{L^2}^2 ds + \epsilon \int_{0}^{t}\|\na (a\theta^{\epsilon})\|_{L^2}^2 ds 
\le \Gamma_4(1+t)^{\frac{3}{2}}  
\ee for any $t \ge 0$, where $\Gamma_4$ is a constant that depends only on the initial data.

{\bf{Step 4. Fractional  moment bounds for $\alpha \in (0,1)$.}} Let $b(x) =(1+|x|^2)^{\frac{1}{2}+\eta}$. The $L^2$ norm of $b \theta^{\epsilon}$ evolves according to 
\be \label{bthetaevolution}
\begin{aligned}
&\frac{1}{2}\frac{d}{dt} \|b\theta^{\epsilon}\|_{L^2}^2 + \|\Lambda^{\frac{\alpha}{2}} (b \theta^{\epsilon})\|_{L^2}^2
\\&\quad= - \int_{\R^2} bJ_{\epsilon}u^{\epsilon} \cdot \na \theta^{\epsilon} b\theta^{\epsilon} dx + \int_{\R^2} (\Lambda^{\alpha}(b\theta^{\epsilon}) - b\Lambda^{\alpha}\theta^{\epsilon}) b\theta^{\epsilon} dx
+ \epsilon \int_{\RR^2} b\Delta \theta^{\epsilon} b\theta^{\epsilon} dx.
\end{aligned}
\ee In order to control the nonlinear term in $u^{\epsilon}$, we write
\be 
b \na \theta^{\epsilon} = \na (b\theta^{\epsilon}) - \theta^{\epsilon} \na b = \na (b\theta^{\epsilon}) - (1+2\eta) \frac{x}{(1+|x|^2)^{\frac{1}{2}-\eta}} \theta^{\epsilon},
\ee (where $\eta \in (0,1/2)$), integrate by parts, use the divergence-free condition obeyed by $J_{\epsilon}u^{\epsilon}$, apply H\"older's inequality, and employ the continuous Sobolev embedding of $\dot{H}^{\frac{\alpha}{2}}$ into $L^{\frac{4}{2-\alpha}}$ to obtain the bound 
\be 
\begin{aligned}
  \left|\int_{\R^2} bJ_{\epsilon}u^{\epsilon} \cdot \na \theta^{\epsilon} b\theta^{\epsilon} dx\right|
  &= \left|\int_{\R^2} (J_{\epsilon}u^{\epsilon} \cdot \na b) b(\theta^{\epsilon})^2 dx\right|
  \le \int_{\R^2} |J_{\epsilon}u^{\epsilon}||\sqrt{1+|x|^2}\theta^{\epsilon}| |b\theta^{\epsilon}| dx
  \\&\le C\|u^{\epsilon}\|_{L^{\frac{4}{\alpha}}} \|b \theta^{\epsilon}\|_{L^{\frac{4}{2-\alpha}}} \|a\theta^{\epsilon}\|_{L^2}
\le C\|u^{\epsilon}\|_{L^{\frac{4}{\alpha}}} \|\Lambda^{\frac{\alpha}{2}}(b \theta^{\epsilon})\|_{L^2} \|a\theta^{\epsilon}\|_{L^2}
\\&\le \frac{1}{4} \|\Lambda^{\frac{\alpha}{2}}(b \theta^{\epsilon})\|_{L^2}^2 + \Gamma_5   \|a\theta^{\epsilon}\|_{L^2}^2 
\end{aligned}
\ee where $\Gamma_5$ is a constant that depends only on the initial data. More precisely, $\Gamma_5 = \|\theta_0\|_{L^{\frac{8}{\alpha}}}^4$ if $u^{\epsilon}= \tilde{T}(\theta^{\epsilon}, \theta^{\epsilon})$ and $\Gamma_5 = \|\theta_0\|_{L^{\frac{4}{\alpha}}}^2$ if $u^{\epsilon}= T\theta^{\epsilon}$. As for the commutator, it holds that 
\be \label{trickyterm}
\begin{aligned}
    &\left|\int_{\R^2} (\Lambda^{\alpha}(b\theta^{\epsilon}) - b\Lambda^{\alpha}\theta^{\epsilon}) b\theta^{\epsilon} dx \right|
    \le \|b \theta^{\epsilon}\|_{L^{2}}
    \|\Lambda^{\alpha}(b\theta^{\epsilon}) - b\Lambda^{\alpha}\theta^{\epsilon}\|_{L^2}
    \\&\le r\|b\theta^{\epsilon}\|_{L^2}^2 + r\|\theta^{\epsilon}\|_{L^1}\|b\theta^{\epsilon}\|_{L^2} + Cr^{-N} \|\theta^{\epsilon}\|_{L^2}\|b\theta^{\epsilon}\|_{L^2} 
    \\&\le 2r\|b \theta^{\epsilon}\|_{L^2}^2 + Cr\|\theta^{\epsilon}\|_{L^1}^2 + Cr^{-2N-1}\|\theta^{\epsilon}\|_{L^2}^2
\end{aligned}
\ee  for any $r = r(t) > 0$ as a consequence of \eqref{corre}. The application of Proposition \ref{double} restricts this case to values of $\alpha$ in the range $0 < \alpha < 1$. 
Choosing $r(t) = (1+t)^{-\frac{3}{2}}$ gives
\be \label{trickyterm}
\begin{aligned}
    &\left|\int_{\R^2} (\Lambda^{\alpha}(b\theta^{\epsilon}) - b\Lambda^{\alpha}\theta^{\epsilon}) b\theta^{\epsilon} dx \right|
    \\&\le 2(t+1)^{-\frac{3}{2}}\|b \theta^{\epsilon}\|_{L^2}^2 + C\|\theta^{\epsilon}\|_{L^1}^2 + C(t+1)^{3N+\frac{3}{2}}\|\theta^{\epsilon}\|_{L^2}^2.
\end{aligned}
\ee  
The $\epsilon$-regularizing term can be treated via integration by parts, as follows,
\be 
\begin{aligned}
    \epsilon \int_{\RR^2} b\Delta \theta^{\epsilon} b\theta^{\epsilon} dx 
    = -\epsilon \|\na (b {\theta}^{\epsilon})\|_{L^2}^2 - 2\epsilon \int_{\RR^2} \na b \cdot \na \theta^{\epsilon} b \theta^{\epsilon} dx - \epsilon \int_{\RR^2} \Delta b \theta^{\epsilon} b \theta^{\epsilon} dx, 
\end{aligned}
\ee where the last two terms can be bounded by 
\be 
\begin{aligned}
    &-2\epsilon \int_{\RR^2} \na b \cdot \na \theta^{\epsilon} b \theta^{\epsilon} dx - \epsilon \int_{\RR^2} \Delta b \theta^{\epsilon} b \theta^{\epsilon} dx
    \\&\quad\quad\le C\epsilon \int_{\R^2} |a \na \theta^{\epsilon}| |b\theta^{\epsilon}| dx
    + C\epsilon \int_{\R^2} |\theta^{\epsilon}| |b \theta^{\epsilon}| dx
     \\&\quad\quad\le C\epsilon \int_{\R^2} (|\na (a \theta^{\epsilon})| + |\na a| |\theta^{\epsilon}|) |b\theta^{\epsilon}| dx
    + C\epsilon \int_{\R^2} |\theta^{\epsilon}| |b \theta^{\epsilon}| dx
    \\&\quad\quad\le C\epsilon \left(\|\na (a\theta^{\epsilon})\|_{L^2} +\|\theta^{\epsilon}\|_{L^2} \right) \|b\theta^{\epsilon}\|_{L^2}
    \\&\quad\quad\le C\epsilon^2 (t+1)^{\frac{3}{2}} \left(\|\na (a\theta^{\epsilon})\|_{L^2}^2 +\|\theta^{\epsilon}\|_{L^2}^2 \right) + (t+1)^{-\frac{3}{2}}\|b\theta^{\epsilon}\|_{L^2}^2.
\end{aligned}
\ee Here we computed  
\be 
\na b = (1+2\eta) \frac{x}{(1+|x|^2)^{\frac{1}{2}-\eta}}, \hspace{1cm} \Delta b = \frac{2(1+2\eta)}{(1+|x|^2)^{\frac{1}{2}-\eta}} + (2\eta - 1)(1+2\eta) \frac{|x|^2}{(1+|x|^2)^{\frac{3}{2}-\eta}}, 
\ee so that 
\be 
|\na b| \le Ca, \hspace{1cm} |\Delta b| \le C
\ee on $\R^2$.
Combining the above estimates and using the uniform boundedness of the regularized solutions in $L^1$ and $L^2$ yield
\be 
\begin{aligned}
\frac{d}{dt} \|b\theta^{\epsilon}\|_{L^2}^2 
&\le \Gamma_5\|a\theta^{\epsilon}\|_{L^2}^2 + \Gamma_6(t+1)^{3N+ \frac{3}{2}} 
\\&+ C\epsilon^2 (t+1)^{\frac{3}{2}} \left(\|\na (a\theta^{\epsilon})\|_{L^2}^2 +\|\theta^{\epsilon}\|_{L^2}^2 \right)
+  C(t+1)^{-\frac{3}{2}}\|b\theta^{\epsilon}\|_{L^2}^2.
\end{aligned}
\ee The incorporation of the time decay $(t+1)^{-\frac{3}{2}}$ as a coefficient to $\|b\theta^{\epsilon}\|_{L^2}^2$ is implemented so that the integrating factor $e^{\int_{0}^{t} \frac{1}{(s+1)^{3/2}} ds}$ is guaranteed to be uniformly bounded in time. 

Finally, we integrate in time, use the bounds \eqref{ppnorm}, \eqref{atheta}, and \eqref{epsat} established in the previous steps,  apply Gronwall's inequality, and employ the weighted commutator estimate \eqref{molcom1} to infer that
\be 
\|b\theta^{\epsilon}(t)\|_{L^2}^2 
\le \Gamma_7(1+t)^{m} 
\ee  for any $t \ge 0$, where $\Gamma_7$ is a positive constant that depends only on the initial data and $m$ is a positive constant that depends on $\alpha$ and $\eta$. 
 
{\bf{Step 5. Fractional moment uniform bounds for $\alpha \in [1,2)$.}} In this case, $b \theta^{\epsilon}$ still evolves according to \eqref{bthetaevolution} and all the terms can be estimated similarly except for the fractional diffusion term. To reapply the commutator estimates established in Corollary \ref{coro}, we introduce a new decomposition technique that reduces the dissipative power $\alpha$ to half its value, shifting it from $[1,2)$ to $(0,1)$. In fact, using the semigroup property obeyed by fractional powers of the Laplacian, we have 
\be 
\int_{\R^2} b\Lambda^{\alpha} \theta^{\epsilon} b \theta^{\epsilon} dx
= \mathcal{A}_1 + \mathcal{A}_2
\ee where 
\be 
\mathcal{A}_1 = \int_{\R^2} \left(b\Lambda^{\frac{\alpha}{2}} (\Lambda^{\frac{\alpha}{2}} \theta^{\epsilon}) - \Lambda^{\frac{\alpha}{2}} (b\Lambda^{\frac{\alpha}{2}}\theta^{\epsilon}) \right) b \theta^{\epsilon} dx
\ee and 
\be 
\mathcal{A}_2 = \int_{\R^2} \Lambda^{\frac{\alpha}{2}} (b\Lambda^{\frac{\alpha}{2}}\theta^{\epsilon})  b \theta^{\epsilon} dx.
\ee We further use the self-adjointness property of $\Lambda^{\frac{\alpha}{2}}$ to rewrite $\mathcal{A}_2$ as 
\be 
\mathcal{A}_2 = \int_{\R^2} b\Lambda^{\frac{\alpha}{2}}\theta^{\epsilon}  \Lambda^{\frac{\alpha}{2}} (b \theta^{\epsilon}) dx
\ee and decompose it into the sum of two terms, $\mathcal{A}_{2,1}$ and $\mathcal{A}_{2,2}$ given by 
\be 
\mathcal{A}_{2,1} = \int_{\R^2}  \left(b\Lambda^{\frac{\alpha}{2}}\theta^{\epsilon} - \Lambda^{\frac{\alpha}{2}} (b\theta^{\epsilon})  \right)\Lambda^{\frac{\alpha}{2}} (b \theta^{\epsilon}) dx
\ee and 
\be 
\mathcal{A}_{2,2} = \int_{\R^2} \Lambda^{\frac{\alpha}{2}} (b\theta^{\epsilon})  \Lambda^{\frac{\alpha}{2}} (b \theta^{\epsilon}) dx
\ee reaching 
\be 
\int_{\R^2} b\Lambda^{\alpha} \theta^{\epsilon} b\theta^{\epsilon} dx 
= \|\Lambda^{\frac{\alpha}{2}} (b\theta^{\epsilon})\|_{L^2}^2
+ \mathcal{A}_1 + \mathcal{A}_{2,1}.
\ee We note that the terms $\mathcal{A}_1$ and $\mathcal{A}_{2,1}$ involve commutators that can be dealt with using Proposition \ref{double} and Corollary \ref{coro} since the involved fractional powers are now in $(0,1)$. Indeed, $\mathcal{A}_{1}$ can be bounded using \eqref{prop15}  as follows, 
\be 
\begin{aligned}
&|\mathcal{A}_1| 
\le \|b\theta^{\epsilon}\|_{L^{2}}   \|b\Lambda^{\frac{\alpha}{2}} (\Lambda^{\frac{\alpha}{2}} \theta^{\epsilon}) - \Lambda^{\frac{\alpha}{2}} (b\Lambda^{\frac{\alpha}{2}}\theta^{\epsilon})\|_{L^2}
\\&\le C\| b\theta^{\epsilon}\|_{L^2} \left(r^2\|\theta^{\epsilon}\|_{L^1} + r^2 \|b\theta^{\epsilon}\|_{L^2} + Cr^{1-N}\|\theta^{\epsilon}\|_{L^2} + Cr\|\Lambda^{\frac{\alpha}{2}} (b\theta^{\epsilon})\|_{L^2} + Cr^{-N}\|\Lambda^{\frac{\alpha}{2}}\theta^{\epsilon}\|_{L^2}\right)
\\&\le \frac{1}{4}\|\Lambda^{\frac{\alpha}{2}}(b\theta^{\epsilon})\|_{L^2}^2 + Cr^2\|b\theta^{\epsilon}\|_{L^2}^2 + C\left(r^2 \|\theta^{\epsilon}\|_{L^1}^2 + r^{-2N} \|\theta^{\epsilon}\|_{L^2}^2 + r^{-2N-2}\|\Lambda^{\frac{\alpha}{2}}\theta^{\epsilon}\|_{L^2}^2  \right).
\end{aligned}
\ee for any $r > 0$. Choosing $r(t) = (t+1)^{-\frac{3}{4}}$ gives 
\be 
\begin{aligned}
&|\mathcal{A}_1| 
\le \|b\theta^{\epsilon}\|_{L^{2}}   \|b\Lambda^{\frac{\alpha}{2}} (\Lambda^{\frac{\alpha}{2}} \theta^{\epsilon}) - \Lambda^{\frac{\alpha}{2}} (b\Lambda^{\frac{\alpha}{2}}\theta^{\epsilon})\|_{L^2}
\le \frac{1}{4}\|\Lambda^{\frac{\alpha}{2}}(b\theta^{\epsilon})\|_{L^2}^2 + C(t+1)^{-\frac{3}{2}} \|b\theta^{\epsilon}\|_{L^2}^2 
\\&\quad\quad\quad\quad+ C\left( \|\theta^{\epsilon}\|_{L^1}^2 + (t+1)^{\frac{3N}{2}}\|\theta^{\epsilon}\|_{L^2}^2 + (t+1)^{\frac{3N}{2}+ \frac{3}{2}}\|\Lambda^{\frac{\alpha}{2}}\theta^{\epsilon}\|_{L^2}^2  \right).
\end{aligned}
\ee As for $\mathcal{A}_{2,1}$, it can be estimated using \eqref{prop14}, as in \eqref{trickyterm}, reaching
\be 
\begin{aligned}
|\mathcal{A}_{2,1}|
&\le  2(t+1)^{-\frac{3}{2}}\|b \theta^{\epsilon}\|_{L^2}^2 + C\|\theta^{\epsilon}\|_{L^1}^2 + C(t+1)^{3N+\frac{3}{2}}\|\theta^{\epsilon}\|_{L^2}^2.
\end{aligned}
\ee  As a consequence, we deduce the following differential inequality 
\be 
\begin{aligned}
&\frac{d}{dt} \|b\theta^{\epsilon}\|_{L^2}^2 \le \Gamma_{8} \|a\theta^{\epsilon}\|_{L^2}^2 + \Gamma_9(t+1)^{n_1} + C(t+1)^{n_2}\|\Lambda^{\frac{\alpha}{2}} \theta^{\epsilon}\|_{L^2}^2 
\\&
+ C\epsilon^2 (t+1)^{\frac{3}{2}} \left(\|\na (a\theta^{\epsilon})\|_{L^2}^2 +\|\theta^{\epsilon}\|_{L^2}^2 \right)
+  C(t+1)^{-\frac{3}{2}}\|b\theta^{\epsilon}\|_{L^2}^2,
\end{aligned}
\ee where $n_1$ and $n_2$ are positive constants depending on $\alpha$ and $\eta$. Integrating in time, using the bounds \eqref{ppnorm}, \eqref{atheta}, and \eqref{epsat} and the commutator estimate \eqref{molcom1}, and applying Gronwall's inequality, we infer that 
\be 
\|b\theta^{\epsilon}(t)\|_{L^2}^2 
\le \Gamma_{10}(1+t)^{m} 
\ee  for any $t \ge 0$, where $\Gamma_{10}$ is a positive constant that depends only on the initial data and $m$ is a positive constant that depend on $\alpha$ and $\eta$.  
\end{proof}

\section{Proof of Theorem \ref{t1}} \label{s4}

In this section, we prove Theorem \ref{t1}. The uniform $L^p$ estimates, uniform moment bounds, and uniform boundedness of the regularized solutions $\theta^{\epsilon}$ in $L^2(0, T; H^{\frac{\alpha}{2}}(\R^2))$ follow directly from Proposition \ref{propsec3}. 

The two-sided uniform relative entropy bounds are derived based on the following two nontrivial major steps:

{\bf{Step 1. Upper bounds.}} We multiply the regularized equation \eqref{moll} by $\log \theta^{\epsilon}$ and we integrate over $\R^2$. Several integrations by parts give rise to the following cancellation laws 
\be 
\int_{\R^2} \pa_t \theta^{\epsilon} dx = 0, \hspace{1cm} \int_{\R^2} J_{\epsilon} u^{\epsilon} \cdot \na \theta^{\epsilon} \log \theta^{\epsilon} = 0,
\ee and the dissipative identity 
\be 
-\epsilon \int_{\R^2} \Delta \theta^{\epsilon} \log \theta^{\epsilon} dx = 4\epsilon \|\na \sqrt{\theta^{\epsilon}}\|_{L^2}^2,
\ee yielding the energy evolution 
\be \label{Nin2}
\frac{d\mathcal{E}^{\epsilon} }{dt} + \kappa \int_{\R^2} \log \theta^{\epsilon} \Lambda^{\alpha}\theta^{\epsilon} dx + 4\epsilon \|\na \sqrt{\theta^{\epsilon}}\|_{L^2}^2= 0.
\ee Since $\frac{\theta^{\epsilon}}{\|\theta_0\|_{L^1}}dx$ is a probability measure and $\log x$ is concave, it follows from Jensen's inequality that 
\be 
\begin{aligned}
    \mathcal{E}^{\epsilon}
    = \|\theta_0\|_{L^1} \int_{\R^2} \frac{\theta^{\epsilon}}{\|\theta_0\|_{L^1}}\log \theta^{\epsilon} dx
    \le \|\theta_0\|_{L^1} \log \int_{\R^2} \frac{(\theta^{\epsilon})^2}{\|\theta_0\|_{L^1}} dx.
\end{aligned}
\ee In view of the continuous Sobolev embedding of $\dot{H}^{\frac{\alpha}{2}}$ into $L^{\frac{4}{2-\alpha}}$ and the uniform boundedness of $\theta^{\epsilon}$ in $L^{\frac{4}{\alpha}}$, we have
\be 
\begin{aligned}
\int_{\R^2} (\theta^{\epsilon})^2 dx 
&= \int_{\R^2} \sqrt{\theta^{\epsilon}} \sqrt{\theta^{\epsilon}} \theta^{\epsilon} dx
\le \|\sqrt{\theta^{\epsilon}}\|_{L^{\frac{4}{2-\alpha}}} \|\sqrt{\theta^{\epsilon}}\|_{L^2} \|\theta^{\epsilon}\|_{L^{\frac{4}{\alpha}}}
\\&\le \|\theta_0\|_{L^1}^{\frac{1}{2}}\|\theta_0\|_{L^{\frac{4}{\alpha}}} \|\Lambda^{\frac{\alpha}{2}}\sqrt{\theta^{\epsilon}}\|_{L^{2}}.
\end{aligned}
\ee Consequently, we obtain the fractional logarithmic inequality 
\be \label{fraclog}
\mathcal{E}^{\epsilon} \le \frac{\|\theta_0\|_{L^1}}{2}\log \left(\frac{\|\theta_0\|_{L^{\frac{4}{\alpha}}}^2 \|\Lambda^{\frac{\alpha}{2}}\sqrt{\theta^{\epsilon}}\|_{L^{2}}^2 }{\|\theta_0\|_{L^1}} \right).
\ee For each $\epsilon > 0$, we denote by $\mathcal{N}^{\epsilon}$ the following functional,
\be 
\mathcal{N}^{\epsilon} = \exp \left\{-2 \mathcal{E}^{\epsilon}/ \|\theta_0\|_{L^1} \right\}.
\ee On the one hand, it follows from \eqref{fraclog} that $\mathcal{N}^{\epsilon}$ obeys the inequality 
\be \label{inn3}
\mathcal{N}^{\epsilon} \|\Lambda^{\frac{\alpha}{2}} \sqrt{\theta^{\epsilon}}\|_{L^2}^2 \ge \frac{\|\theta_0\|_{L^1}}{\|\theta_0\|_{L^{\frac{4}{\alpha}}}^2}. 
\ee On the other hand, differentiating $\mathcal{N}^{\epsilon}$ with respect to time gives
\be \label{Nin}
\frac{d\mathcal{N}^{\epsilon}}{dt} = -\frac{2}{\|\theta_0\|_{L^1}} \mathcal{N}^{\epsilon} \frac{d\mathcal{E}^{\epsilon}}{dt}.
\ee Substituting the time derivative of $\mathcal{E}^{\epsilon}$ by its expression \eqref{Nin2}, we deduce that 
\be 
\begin{aligned}
    \frac{d\mathcal{N}^{\epsilon}}{dt}  
    &= \frac{2}{\|\theta_0\|_{L^1}} \mathcal{N}^{\epsilon} \left(4\epsilon \|\na \sqrt{\theta^{\epsilon}}\|_{L^2}^2 + \kappa \int_{\R^2} \log \theta^{\epsilon}\Lambda^{\alpha} \theta^{\epsilon} dx\right) 
    \\&\ge \frac{2}{\|\theta_0\|_{L^1}} \mathcal{N}^{\epsilon} \left(\kappa \int_{\R^2} \log \theta^{\epsilon}\Lambda^{\alpha} \theta^{\epsilon} dx\right).
\end{aligned}
\ee By the generalized Stroock-Varopoulos inequality (Theorem 2.3, \cite{stan2019existence}), the latter integral can be bounded by the fractional Fischer information as follows, 
\be 
\int_{\R^2} \log \theta^{\epsilon}\Lambda^{\alpha} \theta^{\epsilon} dx \ge 4\|\Lambda^{\frac{\alpha}{2}} \sqrt{\theta^{\epsilon}}\|_{L^2}^2.
\ee Hence, it holds that 
\be 
\frac{d\mathcal{N}^{\epsilon}}{dt} \ge \frac{8\kappa}{\|\theta_0\|_{L^1}} \mathcal{N}^{\epsilon} \|\Lambda^{\frac{\alpha}{2}} \sqrt{\theta^{\epsilon}}\|_{L^2}^2,
\ee which boils down to 
\be 
\frac{d\mathcal{N}^{\epsilon}}{dt} \ge 8\kappa \|\theta_0\|_{L^{\frac{4}{\alpha}}}^{-2}
\ee after making use of the \eqref{inn3}. We integrate in time from $0$ to $t$ and obtain 
\be 
\mathcal{N}^{\epsilon}(t) \ge \mathcal{N}^{\epsilon}(0) + 8\kappa \|\theta_0\|_{L^{\frac{4}{\alpha}}}^{-2}t.
\ee Applying the logarithm on both sides, we infer that 
\be 
\mathcal{E}^{\epsilon}(t) 
\le - \frac{\|\theta_0\|_{L^1}}{2} \log \left(\mathcal{N}^{\epsilon}(0) + \frac{8\kappa}{\|\theta_0\|_{L^{\frac{4}{\alpha}}}^2}t \right)
\ee for any $t \ge 0$. 
We note that 
\be 
\begin{aligned}
\mathcal{E}^{\epsilon}(0) 
&= \int_{\R^2} J_{\epsilon}\theta_0 \log J_{\epsilon}\theta_0 dx = \|\theta_0\|_{L^1} \int_{\R^2} \frac{J_{\epsilon}\theta_0}{\|\theta_0\|_{L^1}}\log J_{\epsilon}\theta_0 dx 
\\&\le \|\theta_0\|_{L^1} \log \left(\int_{\R^2} \frac{(J_{\epsilon}\theta_0)^2}{\|\theta_0\|_{L^1}} dx \right)
\le \|\theta_0\|_{L^1} \log  \frac{\|\theta_0\|_{L^2}^2}{\|\theta_0\|_{L^1}}
\end{aligned}
\ee by Jensen's inequality, so $\mathcal{N}^{\epsilon}(0)$ is uniformly bounded from below by 
\be 
\mathcal{N}^{\epsilon}(0) \ge  \exp \left\{-2 \log  \frac{\|\theta_0\|_{L^2}^2}{\|\theta_0\|_{L^1}} \right\} = \frac{\|\theta_0\|_{L^1}^2}{\|\theta_0\|_{L^2}^4}.
\ee Therefore, the regularized entropies are bounded from above by 
\be 
\mathcal{E}^{\epsilon}(t) 
\le - \frac{\|\theta_0\|_{L^1}}{2} \log \left(\frac{\|\theta_0\|_{L^1}^2}{\|\theta_0\|_{L^2}^4}+ \frac{8\kappa}{\|\theta_0\|_{L^{\frac{4}{\alpha}}}^2}t \right)
\ee for any $t \ge 0$.

{\bf{Step 2. Lower bounds.}} In order to derive uniform-in-$\epsilon$ bounds for the regularized entropies, we perform a splitting technique by which we decompose the domain of integration into two parts: $\left\{\theta^{\epsilon} < 1\right\}$ and $\left\{\theta^{\epsilon} \ge 1 \right\}$, as follows,
\be 
\mathcal{E}^{\epsilon} = \int_{\left\{\theta^{\epsilon} < 1\right\}} \theta^{\epsilon}\log \theta^{\epsilon} dx + \int_{\left\{\theta^{\epsilon} \ge 1\right\}} \theta^{\epsilon}\log \theta^{\epsilon} dx. 
\ee Over the domain $\left\{\theta^{\epsilon} \ge 1\right\}$, we have 
\be 
\left|\int_{\left\{\theta^{\epsilon} \ge 1\right\}} \theta^{\epsilon} \log \theta^{\epsilon}dx \right|
\le C\|\theta^{\epsilon}\|_{L^2}^2
\le C\|\theta_0\|_{L^2}^2
\ee in view of the algebraic inequality $|x\log x| \le Cx^2$ that holds for any $x \ge 1$. Now, over the domain $\left\{\theta^{\epsilon} < 1\right\}$, $|\log \theta^{\epsilon}| =  \log \frac{1}{\theta^{\epsilon}}$ and thus 
\be 
\begin{aligned}
\left|\int_{\left\{\theta^{\epsilon} < 1\right\}} \theta^{\epsilon} \log \theta^{\epsilon} dx\right| 
&\le \int_{\left\{\theta^{\epsilon} < 1\right\}} \theta^{\epsilon} \log \frac{1}{\theta^{\epsilon}} dx
= \frac{1}{1-p}\int_{\left\{\theta^{\epsilon} < 1\right\}} \theta^{\epsilon} \log \frac{1}{(\theta^{\epsilon})^{1-p}} dx
\\&\le \frac{1}{1-p}\int_{\R^2} \theta^{\epsilon} \log \left[ \frac{1}{(\theta^{\epsilon})^{1-p}} + \left(1 - \frac{1}{(\theta^{\epsilon})^{1-p}} \right) \chi_{\theta^{\epsilon} \ge 1} \right]dx,
\end{aligned}
\ee where $p$ is a real number lying in the interval $(\frac{1}{1+\eta}, 1)$ and $\chi_{\theta^{\epsilon }\ge1}$ is the characteristic function 
\be 
\chi_{\theta^{\epsilon} \ge 1}(x) = \begin{cases}
1, \hspace{1cm} \theta^{\epsilon}(x) \ge 1,\\ 0, \hspace{1cm} \theta^{\epsilon}(x) < 1. 
\end{cases}
\ee  An application of Jensen's inequality gives
\be 
\begin{aligned}
\left|\int_{\left\{\theta^{\epsilon} < 1\right\}} \theta^{\epsilon} \log \theta^{\epsilon} dx\right|
&\le \frac{1}{1-p}\|\theta_0\|_{L^1} \log \left[\frac{1}{\|\theta_0\|_{L^1}}\int_{\R^2} (\theta^{\epsilon})^{p} + \theta^{\epsilon} \left(1 - \frac{1}{(\theta^{\epsilon})^{p-1}} \right)\chi_{\theta^{\epsilon}\ge 1} dx\right]
\\&\le \frac{1}{1-p}\|\theta_0\|_{L^1} \log \left(1 +\frac{1}{\|\theta_0\|_{L^1}}\int_{\R^2} (\theta^{\epsilon})^{p} dx \right)
\end{aligned}
\ee By making use of the fractional moment bounds derived in Proposition \ref{propsec3}, we bound 
\be 
\int_{\R^2} (\theta^{\epsilon})^{p} dx
\le \left(\int_{\R^2} (\theta^{\epsilon})^2 (1+|x|^2)^{1+2\eta} dx \right)^{\frac{p}{2}} \left(\int_{\R^2} \frac{1}{(1+|x|^2)^{\frac{p(1+2\eta)}{2-p}}} dx \right)^{\frac{2-p}{2p}} \le \Gamma(1+t)^{\frac{mp}{2}}.
\ee In the last inequality, we used the fact that $\frac{p(1+2\eta)}{2-p}>1$.  Putting the above estimates together, we infer that 
\be 
|\mathcal{E}^{\epsilon}(t)|
\le A( 1 + \log (1+t))
\ee for any $t \ge 0$, where $A$ is a constant that depends only on the initial data. Therefore, we obtain the uniform lower bound
\be 
\mathcal{E}^{\epsilon}(t)  \ge - A(1+ \log (1+t)) 
\ee for any $t \ge 0$.

\section{Proof of Theorem \ref{t2}} \label{s5}

In this section, we present the proof of Theorem \ref{t2}. 
Let $T > 0$. The sequence of regularized solutions $\theta^{\epsilon}$ has a subsequence, denoted by $\theta^{\epsilon_k}$, that converges strongly in $L^2(0,T; L^2(\R^2))$ to a weak solution $\theta$ of \eqref{activescalar}, a fact that follows from the uniform bounds established in Theorem \ref{t1} and an application of the Aubin-Lions lemma. Moreover, we have 
\be 
\theta \in L^{\infty}(0, \infty; L^p(\R^2)) \cap L^2(0, \infty; \dot{H}^{\frac{\alpha}{2}}(\R^2))
\ee for any $p \ge 1$ due to the Banach-Alaoglu theorem and the lower semi-continuity of the norms. The above argument is classical and will be omitted. Here we focus on the convergence of the relative entropies. Indeed, we make use of the Vitali Convergence Theorem to show that the sequence of regularized entropies $\mathcal{E}^{\epsilon_k}$ has a subsequence that converges to $\mathcal{E}$ for almost every time $t \ge 0$, and passing to the limit, we deduce that 
\be 
-C_1 \log (C_2+C_3t) \le \mathcal{E}(t) \le - C_4 \log (C_5+C_6t)
\ee for any $t \ge 0$, where $C_1, \dots C_6$  are positive constants that depend only on the initial data. To this end,  we recall that $\theta^{\epsilon_k}$ converges to $\theta$ in $L^2(0,T; L^2(\R^2))$, and thus, it has a subsequence, that will be denoted by $\theta^{\epsilon_k}$ as well, that converges to $\theta$ in $L^2(\R^2)$ for a.e. $t \in [0, T]$. For such a fixed time $t$, define the sequence of functions
\be 
f_k(x) = \theta^{\epsilon_k}(x) \log \theta^{\epsilon_k}(x)
\ee on $\R^2$. First, the convergence of the sequence $\theta^{\epsilon_k}$ to $\theta$ in $L^2(\R^2)$ implies its convergence in measure. Since $\phi(s) = s\ log s$ is continuous for any $s \ge 0$, we deduce that $\phi(\theta^{\epsilon_k})$ converges to $\phi(\theta)$ by the continuous mapping theorem, and consequently, $f_k$ converges to $\theta \log\theta$ in measure. Second, the sequence $f_k$ is uniformly integrable in $L^1$ due to the uniform boundedness of the regularized entropies established in Theorem \ref{t1}. Finally, it holds that 
\be 
\int_{|x| \ge M} |f_k| dx
\le C\int_{|x| \ge M} \frac{1}{1+|x|^2}\left[(1+|x|^2)|\theta^{\epsilon_k}|^2 \right] dx
\le \frac{C\|\sqrt{1+|x|^2}\theta^{\epsilon_k}\|_{L^2}^2}{1+M^2} \le \frac{C\Gamma_t}{1+M^2}
\ee for any $M \ge 0$, and hence 
\be 
\lim\limits_{M \to \infty} \sup\limits_{k \in \N} \int_{|x| \ge M} |f_k| dx = 0. 
\ee As the conditions of the Vitali Convergence Theorem are satisfied, we infer that $f_k$ converges to $\theta \log \theta$ in $L^1(\R^2)$, ending the proof of Theorem \ref{t2}.

\vspace{0.5cm}

{\bf{Acknowledgments.}} E.A. was partially supported by the University Research Board (URB) of the American University of Beirut under Grant No. 104752. 

\vspace{0.5cm}

{\bf{Data Availability Statement.}} The research does not have any associated data.

\vspace{0.5cm}

{\bf{Conflict of Interest.}} The authors declare that they have no conflict of interest.

\bibliographystyle{plain}
\bibliography{ref}

\end{document}